\documentclass{amsart}
\usepackage{amssymb}
\usepackage[utf8x]{inputenc}
\usepackage{hyperref}

\DeclareMathOperator{\im}{Im}
\DeclareMathOperator{\adj}{adj}
\DeclareMathOperator{\supp}{supp}
\newtheorem{theorem}{Theorem}[section]
\newtheorem{lemma}[theorem]{Lemma}
\newtheorem{proposition}[theorem]{Proposition}
\theoremstyle{remark}\newtheorem{remark}[theorem]{Remark}
\numberwithin{equation}{section}
\title[Vertical projections in the Heisenberg group]{An a.e.~lower bound for Hausdorff dimension under vertical projections in the Heisenberg group}
\author{Terence L.~J.~Harris}
\address{Department of Mathematics, University of Illinois, Urbana, IL 61801, U.S.A.}
\email{terence2@illinois.edu}
\thanks{This material is based upon work partially supported by the National Science Foundation under Grant No. DMS-1501041. I would like to thank Burak Erdoğan for advice, and for financial support. I also thank Jeremy Tyson for suggesting the use of Proposition \ref{covering}, which improved the final result.}

\begin{document}
\begin{abstract} An improved a.e.~ lower bound is given for Hausdorff dimension under vertical projections in the first Heisenberg group. 
\end{abstract}
\maketitle
\section{Introduction}  
The aim of this work is to improve the known a.e.~lower bounds for Hausdorff dimension under vertical projections in the Heisenberg group. The average behaviour of Hausdorff dimension under orthogonal projections in Euclidean space was first explored by Marstrand in 1954 \cite{marstrand}; many developments and generalisations have occurred since (see e.g.~\cite{mattila3,peres,falconer,balogh4,mattila4}). An effort began in \cite{balogh} and \cite{balogh2} towards understanding the behaviour of Hausdorff dimension under projections in the Heisenberg group, which was further developed in \cite{hovila} and \cite{fassler} (see also \cite{balogh5}). One important open problem that remains is determining the a.e.~behaviour of Hausdorff dimension under ``vertical projections''. 

All definitions relevant to this work will be restated here; further background is available in \cite{balogh,balogh2}. Let $\mathbb{H}$ be the first Heisenberg group, which as a set will be identified with $\mathbb{R}^3 = \mathbb{C} \times \mathbb{R}$. The assumed convention for the group law on $\mathbb{H}$ is
\begin{align*} (z,t) \ast (\zeta,\tau) &= \left(z+\zeta, t+ \tau + 2\im{\left(z\overline{\zeta} \right)} \right) \\
&= \left(z+\zeta, t+ \tau - 2 z \wedge \zeta \right),\end{align*}
where $\wedge: \mathbb{R}^2 \times \mathbb{R}^2 \to \mathbb{R}$ is the standard wedge product on $\mathbb{R}^2$, given by 
\[ (x_1,y_1) \wedge (x_2,y_2) = x_1y_2-x_2y_1. \]
 Define $\|(z,t) \|_{\mathbb{H}} = \left( |z|^4 + t^2 \right)^{1/4}$. The group $\mathbb{H}$ is a metric space when equipped with the left invariant metric $d_{\mathbb{H}}$, called the Korányi metric, defined by
\begin{align} \notag d_{\mathbb{H}}((z,t), (\zeta,\tau) ) &= \left\| (\zeta,\tau)^{-1} \ast (z,t) \right\|_{\mathbb{H}}  \\
\label{secondcomp} &= \left( |z-\zeta|^4 + |t-\tau-2z \wedge \zeta|^2 \right)^{1/4}; \end{align}
see \cite{cygan} for a proof of the triangle inequality. On any compact set, this metric is bi-Lipschitz equivalent to the usual Carnot-Carathéodory metric on $\mathbb{H}$ (see \cite{balogh}).

 For a given metric space, Hausdorff dimension is defined through the underlying distance, which for the Heisenberg group will always be the Korányi metric. Hausdorff dimension is invariant under a bi-Lipschitz change of the metric, so the main results given here will hold for the Carnot-Carathéodory metric too. Under either of these metrics, $\mathbb{H}$ has Hausdorff dimension 4.

For each $\theta \in [0,\pi)$, let \[   V_{\theta} = \{ \lambda e^{i \theta} \in \mathbb{C}: \lambda \in \mathbb{R} \},  \quad V_{\theta}^{\perp} = \{ \lambda ie^{i \theta} \in \mathbb{C}: \lambda \in \mathbb{R} \}. \]
Define the horizontal subgroups $\mathbb{V}_{\theta} \subseteq \mathbb{H}$ by
\[ \mathbb{V}_{\theta} = \{ (\lambda e^{i \theta},0) \in \mathbb{C} \times \mathbb{R}: \lambda \in \mathbb{R} \}, \]
and define the vertical subgroups $\mathbb{V}_{\theta}^{\perp} \subseteq \mathbb{H}$ by
\[ \mathbb{V}^{\perp}_{\theta} = \{(\lambda_1 i e^{i \theta}, \lambda_2)  \in \mathbb{C} \times \mathbb{R}: \lambda_1, \lambda_2 \in \mathbb{R} \}, \]
so that $\mathbb{V}_{\theta}^{\perp}$ is the Euclidean orthogonal complement of $\mathbb{V}_{\theta}$ in $\mathbb{R}^3$. Let $\pi_{V_{\theta}}: \mathbb{C} \to \mathbb{C}$ be the orthogonal projection onto $V_{\theta}$, and similarly let $\pi_{V_{\theta}^{\perp}}$ be projection onto $V_{\theta}^{\perp}$. The horizontal and vertical projections $P_{\mathbb{V}_{\theta}} : \mathbb{H} \to \mathbb{V}_{\theta}$ and  $P_{\mathbb{V}^{\perp}_{\theta}}: \mathbb{H} \to \mathbb{V}^{\perp}_{\theta}$ are defined for each $\theta \in [0,\pi)$ by  
\[P_{\mathbb{V}_{\theta}}(z,t) = \left(\pi_{V_{\theta}}(z) ,0 \right), \quad  P_{\mathbb{V}^{\perp}_{\theta}}(z,t) = \left(\pi_{V_{\theta}^{\perp}}(z), t - 2\pi_{V_{\theta}}(z) \wedge \pi_{V_{\theta}^{\perp}}(z) \right). \]
The term ``projection'' and the formulas for $P_{\mathbb{V}_{\theta}}$, $P_{\mathbb{V}^{\perp}_{\theta}}$ come from the unique way of writing an element \[ (z,t) = P_{\mathbb{V}_{\theta}^{\perp}}(z,t) \ast  P_{\mathbb{V}_{\theta}}(z,t), \]
 as a product of an element of $\mathbb{V}^{\perp}_{\theta}$ on the left, with an element of $\mathbb{V}_{\theta}$ on the right.  

In \cite[Theorem 1.4]{balogh} it was shown that for any Borel (or analytic) set $A \subseteq \mathbb{H}$,
\begin{equation} \label{baloghb} \dim P_{\mathbb{V}^{\perp}_{\theta}} A \geq \begin{cases} \dim A &\text{ if } 0 \leq \dim A < 1 \\
1 &\text{ if } 1 \leq \dim A < 3 \\
2  \dim A -5 &\text{ if } 3 \leq \dim A \leq 4, \end{cases} \end{equation}
for a.e.~$\theta \in [0,\pi)$, and it was conjectured that the lower bound $\dim P_{\mathbb{V}^{\perp}_{\theta}} A \geq \dim A$ actually holds in the larger range $0 \leq \dim A \leq 3$. The upper limit of 3 is necessary since the vertical subgroups $\mathbb{V}_{\theta}^{\perp}$ have Hausdorff dimension 3. In \cite{fassler}, Fässler and Hovila proved
\begin{equation} \label{fasslerb} \dim P_{\mathbb{V}^{\perp}_{\theta}} A \geq 1 + \frac{(\dim A-1)(\dim A-2)}{32(\dim A)^2}, \quad \text{for a.e.~$\theta \in [0,\pi)$,} \quad \dim A >2, \end{equation}
which improved \eqref{baloghb} in the range $2 < \dim A < 3.00348$ (approximately). The main result of this work is the following lower bound.
\begin{theorem} \label{maintheorem} If $A \subseteq \mathbb{H}$ is an analytic set with $\dim A >1$, then 
\[ \dim P_{\mathbb{V}^{\perp}_{\theta}} A  \geq    \begin{cases} \frac{\dim A}{2} &\text{ if } \dim A \in \left(2, \frac{5}{2} \right] \\
\frac{ \dim A(\dim A + 2)}{4\dim A -1}  &\text{ if } \dim A \in \left( \frac{5}{2} , 4 \right],  \end{cases}    \]
for a.e.~$\theta \in [0,\pi)$.
\end{theorem}
This improves \eqref{baloghb} and \eqref{fasslerb} in the range $2  < \dim A < \frac{12+\sqrt{109}}{7}$ (the upper bound is roughly 3.2). The proof employs some of the techniques used by Orponen and Venieri in \cite{venieri} for restricted families of projections in $\mathbb{R}^3$; the main difficulty in adapting this to the Heisenberg setting lies in finding a substitute for Marstrand's ``Three Circles Lemma'' (see \cite[Lemma 3.2]{wolff}). The key observation is that if two points $(z,t), (\zeta,\tau) \in \mathbb{H}$ have their vertical projections close to each other for some angle $\theta$, then the second component $|t-\tau-2z \wedge \zeta|$ of the Korányi distance between them is small, and the latter function is independent of the angle $\theta$. Furthermore, if $(\zeta_i,\tau_i) \in \mathbb{H}$ are three points such that the $\zeta_i$'s are not collinear, then there is at most one point $(z,t) \in \mathbb{H}$ such that $(t-\tau_i-2z \wedge \zeta_i)$ vanishes for all $i$'s simultaneously (this follows from a calculation involving linear algebra). A quantitative version of this constraint is used in the proof of Lemma \ref{compromise}.

\subsection{Notation and preliminaries}
Given two measurable spaces $X$ and $Y$, a measure $\nu$ on $X$ and a measurable function $f: X \to Y$, the pushforward $f_{\#}\nu$ of $\nu$ under $f$ is a measure on $Y$, defined by $f_{\#}\nu(E) = \nu(f^{-1}(E))$ for each measurable set $E \subseteq Y$.

For a real number $t$, let $\lceil t \rceil$ denote the least integer greater than or equal to $t$.

The relation $x \lesssim y$ will mean that $x \leq Cy$ for some constant $C>0$, where $C$ will sometimes depend on objects that are fixed throughout a given proof. 

 Let $|x|$ denote the Euclidean norm of an element $x \in \mathbb{R}^n$. The Euclidean distance $|x-y|$ between $x$ and $y$ may also be denoted by $d_E(x,y)$. For $x \in \mathbb{R}^3$ and $r>0$, let $B_E(x,r)$ and $B_{\mathbb{H}}(x,r)$ be the Euclidean and Korányi balls around $x$ of radius $r$. The following local Hölder condition from \cite{balogh3} shows that $\left(\mathbb{R}^3, d_E\right)$ and $\left( \mathbb{H}, d_{\mathbb{H}} \right)$ are homeomorphic; this fact is used implicitly throughout.
\begin{lemma}[{\cite[Lemma 2.1]{balogh3}}] \label{holder} For any $R>0$, there exists a positive constant $c=c(R)>0$ such that
\[ c^{-1}|v-w|  \leq d_{\mathbb{H}}(v,w) \leq c |v-w|^{1/2}, \]
for all $v,w \in \mathbb{R}^3$ with $|v|,|w| \leq R$.  
\end{lemma}
In some situations, the following proposition gives a covering of a Euclidean ball by Korányi balls which is more efficient than the single ball covering implied by the previous lemma.
\begin{proposition}[{\cite[Proposition 3.4]{balogh3}}] \label{covering} For any $R>0$, there exists a positive integer $N=N(R)>0$ such that any Euclidean ball $B_E(v,r)$ with $|v| \leq R$ and $r \in (0,1)$ can be covered by at most $\left\lceil N/r\right\rceil$ Korányi balls of radius $r$. 
\end{proposition}

The following version of Frostman's Lemma provides a characterisation of Hausdorff dimension for {\it analytic} sets (see \cite{howroyd,mattila3} for a proof). To state it, a subset $A$ of a complete separable metric space $X$ is called analytic if $A$ is the continuous image of a Borel set $B \subseteq Y$, for some complete separable metric space $Y$. In particular, every Borel subset of $X$ is analytic. 
\begin{lemma} \label{frostman} Let $X$ be a complete separable metric space, let $A$ be an analytic subset of $X$ and let $s>0$. If there exists a nonzero finite Borel measure $\nu$ on $A$ and a constant $C$ such that 
\begin{equation} \label{frostmanc} \nu(B(x,r)) \leq Cr^s \quad \text{for all $r>0$ and $x \in A$,} \end{equation} then $\dim A \geq s$. Conversely, if $\dim A >s$ then there exists a compactly supported, nonzero, finite Borel measure $\nu$ on $A$ satisfying \eqref{frostmanc}.
\end{lemma} 
Measures satisfying \eqref{frostmanc} are sometimes called fractal measures, or ``$s$-Frostman measures'', and \eqref{frostmanc} is sometimes referred to as a ``Frostman condition''.

\section{Proof of lemmas and the main theorem}
Most of this section is devoted to proving the lemmas from which Theorem~\ref{maintheorem} will follow. The first lemma of this section is an abstract version of Lemma 2.5 from \cite{venieri} (see also \cite[Theorem 7.2]{kaenmaki}); the proof is not too different from the Euclidean case, but is included for completeness. In the statement of the lemma, $(\theta,x) \mapsto \pi_{\theta}(x)$ is an arbitrary continuous function, but all statements following the proof of the lemma will specialise to the case where $\pi_{\theta} = P_{\mathbb{V}^{\perp}_{\theta}}$ is a vertical projection on $\mathbb{H}$. The lemma essentially says that: given a fractal measure on a set $A$, if there is a quantitative restriction on how often the pushforward measure under the projection fails an $s$-Frostman condition, then a.e.~the dimension of $\pi_{\theta}(A)$ is at least $s$ (where $s$ may be smaller than $\dim A$, but ideally as close to $\dim A$ as possible). 

\begin{lemma}\label{abstract} Let $X$, $Y$ be metric spaces, with $X$ compact and $Y$ separable. Suppose that $\mu$ is a finite Borel measure on $X$, $\nu$ is a nonzero, finite, compactly supported Borel measure on $Y$, and $(\theta,y) \mapsto \pi_{\theta}(y)$ is a continuous function from $X \times Y$ into $Y$. Given $s>0$, if there exist $\eta, \delta_0 >0$ such that
\begin{equation} \label{assumption} \nu\left\{ y \in Y: \mu\left\{ \theta \in X :  \pi_{\theta \#}\nu (B(\pi_{\theta}(y), \delta) ) \geq \delta^s \right\} \geq \delta^{\eta} \right\} \leq \delta^{\eta}, \end{equation}
for all $\delta \in (0, \delta_0)$, then 
\[ \dim \pi_{\theta}( \supp \nu) \geq s \quad \text{ for $\mu$-a.e.~$\theta \in X$.} \]
\end{lemma}
\begin{remark}
The proof of the lemma necessarily has a few measure-theoretic technicalities; the core part of the proof is the calculation following \eqref{start}. 
\end{remark}
\begin{proof}[Proof of Lemma~\ref{abstract}] Let $\mu$, $\nu$, $\eta$, $\delta_0$, $s$ be given. It is first shown that the sets occurring in \eqref{assumption} are measurable. For fixed $x \in Y$, and any constant $c>0$, the set
\[ S := \left\{ (\theta,y) \in X \times Y : d( \pi_{\theta}(x),\pi_{\theta}(y) ) <c \right\}, \]
is open in $X \times Y$ by continuity. Since $Y$ is separable, the Borel sigma algebra on $X \times Y$ is equal to the one generated by the products of Borel sets \cite[Lemma 6.4.2]{bogachev}, and is therefore contained in the class of $(\mu \times \nu)$-measurable sets, since $\mu$ and $\nu$ are Borel by assumption. Hence $S$ is $(\mu \times \nu)$-measurable. Therefore the function 
\[ f(\theta,y) = \chi_{\pi_{\theta}^{-1} ( B( \pi_{\theta}(x), c ) )}(y), \]
is $(\mu \times \nu)$-measurable, and so the function 
\begin{equation} \label{measurable} \theta \mapsto \int f(\theta,y) \, d\nu(y) = \pi_{\theta \#} \nu(B(\pi_{\theta}(x), c )) \end{equation}
 is $\mu$-measurable in $\theta$ by part (iv) of Fubini's Theorem from \cite{evans}. This proves $\mu$-measurability of the inner part of \eqref{assumption}.

For the outer part of \eqref{assumption}, denoted by
\begin{equation} \label{Zdelta} Z_{\delta} := \left\{ y \in Y: \mu\left\{ \theta \in X :  \pi_{\theta \#} \nu (B (\pi_{\theta}(y), \delta ) ) \geq \delta^s \right\} \geq \delta^{\eta} \right\}, \end{equation}
a similar argument to that for \eqref{measurable} shows that for any $\delta>0$ the function $(\theta,y) \mapsto \pi_{\theta \#} \nu(B(\pi_{\theta}(y), \delta ))$ is $(\mu \times \nu)$-measurable, and hence the function 
\[ y \mapsto \mu\left\{ \theta \in X :  \pi_{\theta \#} \nu (B(\pi_{\theta}(y), \delta) ) \geq \delta^s \right\} \]
is a $\nu$-measurable function of $y$, by part (iii) of Fubini's Theorem from \cite{evans}. This shows that $Z_{\delta}$ is $\nu$-measurable.  

 Since $\nu$ is compactly supported and $Y$ is separable, and since $(\theta,y) \mapsto \pi_{\theta}(y)$ is continuous, to prove the lemma it may be assumed that $Y$ is compact. Let $\epsilon > 0$ and let $E \subseteq X$ be a compact set with
\begin{equation} \label{dimdefn} \dim \pi_{\theta}( \supp \nu) < s-\epsilon \quad \text{ for every $\theta \in E$.} \end{equation}
Any finite Borel measure on a compact metric space is inner regular, so to establish the lemma it suffices to show that $\mu(E) = 0$. Let $\epsilon' >0$, and choose a positive $\delta_1 < \delta_0/4$ small enough to ensure $\delta_1^{\eta} \leq \epsilon'$. For each $\theta \in E$, using \eqref{dimdefn} let $\left\{B(\pi_{\theta}(z_i(\theta)), \delta_i(\theta) ) \right\}_{i=1}^{\infty}$ be a cover of $\pi_{\theta}(\supp \nu)$ by balls in $Y$ of dyadic radii $\delta_i(\theta)< \delta_1$ such that $\sum_{i=1}^{\infty} \delta_i(\theta)^s < \epsilon'$. It may be assumed that each $z_i(\theta) \in \supp \nu$. 

\begin{sloppypar}For each integer $j$, let
\[ D_{\theta}^j := \bigcup_{ \delta_i(\theta) = 2^{-j} } B\left( \pi_{\theta}(z_i(\theta)), 2^{-j} \right), \] 
and let $D_{\theta,g}^j$ be the subset of $D_{\theta}^j$ defined as the union over those balls $B\left( \pi_{\theta}(z_i(\theta)), 2^{-j} \right)$ in $D_{\theta}^j$ with $\pi_{\theta \#}\nu \left( B\left( \pi_{\theta}(z_i(\theta)), 2^{-j} \right) \right) < 2^{-(j-1)s}$. Let $D_{\theta,b}^j$ be the union of the remaining balls in $D_{\theta}^j$, equivalently $D_{\theta,b}^j = D_{\theta}^j \setminus D_{\theta,g}^j$. As will be shown, it is possible to choose the covers of $\pi_{\theta}(\supp \nu)$ in such a way that the functions 
\begin{equation} \label{measurable2}  \pi_{\theta \#}\nu(B(\pi_{\theta}(z_i(\theta)), c )), \quad   \pi_{\theta \#}\nu\left( D_{\theta}^j \right), \quad  \nu\left(  \pi_{\theta}^{-1} \left(D_{\theta}^j\right)  \cap Z_{\delta}\right), \end{equation}
and 
\begin{equation} \label{measurable3}  \nu\left( \pi_{\theta}^{-1} \left(D_{\theta,g}^j\right)\setminus Z_{\delta}\right), \quad \nu\left( \pi_{\theta}^{-1} \left(D_{\theta,b}^j\right)\setminus Z_{\delta}\right), \end{equation}
 are $\mu$-measurable in $\theta$ on $E$, for any $c,\delta>0$, for every $i$ and for every integer $j$. 

To verify the $\mu$-measurability of \eqref{measurable2} and \eqref{measurable3}, the compactness of $\pi_{\theta}(\supp \nu)$ for each fixed $\theta \in E$ ensures that there is a finite subcollection (not relabelled) of balls $B(\pi_{\theta}(z_i(\theta)), \delta_i(\theta))$ which cover $\pi_{\theta}(\supp \nu)$. The union $U_{\theta}$ of these balls is an open set, and therefore contains an open $\delta'$-neighbourhood $\mathcal{N}_{\delta'}(\pi_{\theta}(\supp \nu))$ of $\pi_{\theta}(\supp \nu)$ for some $\delta' >0$. The compactness of $Y$ (assumed without loss of generality) ensures that the map $(\theta,y) \mapsto \pi_{\theta}(y)$ is uniformly continuous on $X \times Y$, which implies
\[ \pi_{\theta'} (\supp \nu) \subseteq \mathcal{N}_{\delta'}(\pi_{\theta}(\supp \nu)) \subseteq U_{\theta}, \]
for all $\theta'$ in a sufficiently small ball $B_{\theta}$ around $\theta$. Therefore the balls $B\left(\pi_{\theta}(z_i(\theta)\right), \delta_i(\theta))$ form a finite cover of $\pi_{\theta'}(\supp \nu)$ for $\theta' \in B_{\theta}$. The sets $B_{\theta}$ cover $E$ as $\theta$ ranges over $E$, so by compactness of $E$ there is a finite subcollection $\left\{B_{\theta_1}, \dotsc, B_{\theta_N}\right\}$ such that 
\[ E = B_{\theta_1} \cup B_{\theta_2} \setminus B_{\theta_1} \cup \dotsb \cup B_{\theta_N} \setminus \cup_{i=1}^{N-1} B_{\theta_i}. \]
The functions $z_i(\theta)$ and $\delta_i(\theta)$ may then be taken to be constant on each part of this Borel partition of $E$. By the piecewise constant property and the $\mu$-measurability of \eqref{measurable}, the function $\pi_{\theta \#}\nu(B(\pi_{\theta}(z_i(\theta)), c ))$ is $\mu$-measurable for every $i$ and any $c>0$. This proves the $\mu$-measurability of the first function in \eqref{measurable2}. Measurability of the other functions follows from a similar argument to the measurability of \eqref{measurable}, using the piecewise constant property of the $\delta_i(\theta)$'s and the $\nu$-measurability of $Z_{\delta}$. This shows that the covers $\left\{B(\pi_{\theta}(z_i(\theta)), \delta_i(\theta) ) \right\}_{i=1}^{\infty}$ may be chosen to make the functions in \eqref{measurable2} and \eqref{measurable3} $\mu$-measurable over $E$. \end{sloppypar}

For each $\theta \in E$, 
\begin{equation} \label{start} \nu(Y) \leq \sum_{j > \left|\log_2 \delta_1\right|} \pi_{\theta \#}\nu\left( D_{\theta}^j \right),  \end{equation}
by the definition of the cover and the sets $D_{\theta}^j$. Dividing both sides by $\nu(Y) \gtrsim 1$ and integrating over $E$ gives
\begin{align} \notag \mu(E) &\lesssim  \int_E \sum_{j > \left|\log_2 \delta_1\right| }  \pi_{\theta \#}\nu\left( D_{\theta}^j \right) \, d\mu(\theta) \\
\label{int1} &\leq \sum_{j > \left|\log_2 \delta_1\right|} \int_E \nu\left(  \pi_{\theta}^{-1} \left(D_{\theta}^j\right)  \cap Z_{2^{-(j-1)}}\right) \, d\mu(\theta) \\ 
\label{int2} &\quad +   \int_E \sum_{j > \left|\log_2 \delta_1\right|} \nu\left( \pi_{\theta}^{-1} \left(D_{\theta,g}^j\right)\setminus Z_{2^{-(j-1)}}\right) \, d\mu(\theta) \\
\label{int3} &\quad  + \sum_{j > \left|\log_2 \delta_1\right|} \int_{E} \nu\left( \pi_{\theta}^{-1} \left(D_{\theta,b}^j\right)\setminus Z_{2^{-(j-1)}}\right) \, d\mu(\theta). \end{align} 
It remains to bound the integrals in \eqref{int1}, \eqref{int2} and \eqref{int3}. Up to a constant the first sum, in \eqref{int1}, is bounded by $\delta_1^{\eta} \leq \epsilon'$ by the assumption \eqref{assumption} on each $Z_{\delta}$ in the statement of the lemma, and the choice of $\delta_1$. The integral in \eqref{int2} is $\lesssim \epsilon'$ by the condition $\sum_{i=1}^{\infty} \delta_i(\theta)^s \leq \epsilon'$ defining the cover and by the definition of $D_{\theta,g}^j$. It remains to bound \eqref{int3}. For each $j$ the set 
\[ \left\{ (\theta, y) \in E \times Y: \pi_{\theta}(y) \in D_{\theta,b}^j \right\} \]
is $(\mu \times \nu)$-measurable by the piecewise constant property of the defining cover. Applying Fubini's Theorem and then the definition of $D_{\theta,b}^j$ to each integral in \eqref{int3} results in
\begin{align*}  &\sum_{j > \left|\log_2 \delta_1\right|} \int_E \nu\left( \pi_{\theta}^{-1} \left(D_{\theta,b}^j\right)\setminus Z_{2^{-(j-1)}}\right) \, d\mu(\theta) \\
&\quad = \sum_{j > \left|\log_2 \delta_1\right|} \int_{Y \setminus Z_{2^{-(j-1)} } } \mu \Big\{ \theta \in E : \pi_{\theta}(y)  \in D_{\theta,b}^j \Big\} \, d\nu(y) \\
&\quad \leq  \sum_{j > \left|\log_2 \delta_1\right| } \int_{Y \setminus Z_{2^{-(j-1)} } } \mu\Big\{ \theta \in E : \pi_{\theta \#}\nu \left( B\left( \pi_{\theta}(y), 2^{-(j-1)} \right) \right) \geq 2^{-(j-1)s} \Big\} \, d\nu(y) \\
&\quad \lesssim  \sum_{j > \left|\log_2 \delta_1\right|} 2^{-j \eta} \qquad \text{by the definition of $Z_{2^{-(j-1)}}$ in \eqref{Zdelta},} \\
&\quad \lesssim \epsilon', \end{align*}
by the condition $\delta_1^{\eta} \leq \epsilon'$ imposed in the choice of $\delta_1$. Therefore $\mu(E) \lesssim \epsilon'$ with $\epsilon'$ arbitrary, and thus $\mu(E)=0$. This proves the lemma. \end{proof}

The following lemma is a slightly refined version of Lemma 3.5 from \cite{fassler}, see also \cite[Section 4]{balogh2}. The lemma is a kind of transversality condition, which says that in a quantitative sense the paths of two fixed, distinct points under the family of vertical projections pass each other transversally. The proof has only minor adjustments to those in \cite{balogh2,fassler} but is included for completeness. 

\begin{lemma} \label{fasslerlemma} There exists a positive constant $C$ such that for any $v,w \in \mathbb{H}$ and any $\delta>0$, the set
\[ \left\{ \theta \in [0,\pi) : d_{\mathbb{H}}\left( P_{\mathbb{V}^{\perp}_{\theta}}(v), P_{\mathbb{V}^{\perp}_{\theta}}(w) \right) \leq \delta \right\}, \]
is contained in a disjoint union of at most 40 intervals, each of length less than $\frac{C\delta}{d_{\mathbb{H}}(v,w) }$.
\end{lemma}
\begin{proof} Fix $v,w \in \mathbb{H}$ and write $v=(z,t), w=(\zeta,\tau)$. If 
\begin{equation} \label{caseone} |z-\zeta| \geq |t-\tau-2z \wedge \zeta|^{1/2}, \end{equation}
then
\begin{multline} \label{trivialcase} \left\{ \theta \in [0,\pi) : d_{\mathbb{H}}\left( P_{\mathbb{V}^{\perp}_{\theta}}(v), P_{\mathbb{V}^{\perp}_{\theta}}(w) \right) \leq \delta \right\} \\
\subseteq \left\{ \theta \in [0,\pi) : |\pi_{V_{\theta}^{\perp}}(z-\zeta) | \leq \delta \right\}. \end{multline}
By writing $z-\zeta = |z-\zeta|e^{i \phi}$ and rotating so that $\phi = 0$, the right hand side of \eqref{trivialcase} is contained in two intervals of length $\lesssim \frac{\delta}{|z-\zeta|} \lesssim \frac{\delta}{d_{\mathbb{H}}(v,w)}$. This proves the lemma in the case that \eqref{caseone} holds.

If \eqref{caseone} fails, then
\begin{equation} \label{case2} |z-\zeta| < |t-\tau-2z \wedge \zeta|^{1/2}. \end{equation}

Suppose \eqref{case2} holds and $z = \pm \zeta$. Then 
\begin{align} \notag d_{\mathbb{H}}(v,w) &= \left( |z-\zeta|^4 + |t-\tau - 2 z \wedge \zeta |^2 \right)^{1/4} \\
\notag &= \left( |z-\zeta|^4 + |t-\tau |^2 \right)^{1/4} \\
\notag &\leq 2^{1/4} |t-\tau |^{1/2} \\
\notag &=  2^{1/4} \left|t-\tau - 2 \pi_{V_{\theta}}(z) \wedge \pi_{V_{\theta}^{\perp}}(z) + 2\pi_{V_{\theta}}(\zeta) \wedge \pi_{V_{\theta}^{\perp}}(\zeta)  \right|^{1/2} \\
\label{trivial} &\leq 2^{1/4}d_{\mathbb{H}}\left( P_{\mathbb{V}^{\perp}_{\theta}}(v), P_{\mathbb{V}^{\perp}_{\theta}}(w) \right), \end{align}
and so
\[ \left\{ \theta \in [0,\pi) : d_{\mathbb{H}}\left( P_{\mathbb{V}^{\perp}_{\theta}}(v), P_{\mathbb{V}^{\perp}_{\theta}}(w) \right) \leq \delta \right\} \subseteq \left\{ \theta \in [0,\pi) : d_{\mathbb{H}}\left( v, w\right) \leq 2^{1/4}\delta \right\}.\]
The right hand side is $[0,\pi)$ if $d_{\mathbb{H}}\left( v, w\right) \leq 2^{1/4}\delta$, which in this case is an interval of length $\pi \lesssim \frac{\delta}{d_{\mathbb{H}}\left( v, w\right)}$. Otherwise the right hand side is empty and there is nothing to show. This proves the lemma in the case of \eqref{case2} with $z = \pm \zeta$.

It remains to consider the case in which \eqref{case2} holds but $z \neq \pm \zeta$. Let
\[ a = \frac{ t-\tau - 2 z \wedge \zeta}{|z+\zeta| |z-\zeta|}, \quad p = \frac{z-\zeta}{|z-\zeta|}, \quad q = \frac{z+\zeta}{|z+\zeta|}. \]
The identities 
\[ z \wedge \zeta =  \pi_{V_{\theta}}(z) \wedge \pi_{V_{\theta}^{\perp}}(\zeta) -  \pi_{V_{\theta}}(\zeta) \wedge \pi_{V_{\theta}^{\perp}}(z), \]
and 
\[ \pi_{V_{\theta}^{\perp}} (x) \wedge \pi_{V_{\theta}}(y) = x \wedge \pi_{V_{\theta}}(y), \quad x,y \in \mathbb{R}^2, \]
yield
\begin{align} \notag d_{\mathbb{H}}\left( P_{\mathbb{V}^{\perp}_{\theta}}(v), P_{\mathbb{V}^{\perp}_{\theta}}(w) \right) &\geq \left|t-\tau- 2\pi_{V_{\theta}}(z) \wedge \pi_{V_{\theta}^{\perp}}(z)+ 2\pi_{V_{\theta}}(\zeta) \wedge \pi_{V_{\theta}^{\perp}}(\zeta)\right|^{1/2} \\
\notag &= \left|t-\tau- 2z \wedge \zeta +2 \pi_{V_{\theta}^{\perp}}(z-\zeta) \wedge \pi_{V_{\theta}}(z+\zeta) \right|^{1/2} \\
\notag &= \left|t-\tau- 2z \wedge \zeta +2(z-\zeta) \wedge \pi_{V_{\theta}}(z+\zeta) \right|^{1/2} \\ 
\label{lastexpr} &= |z+\zeta|^{1/2}|z-\zeta|^{1/2} \left|a + 2p \wedge \pi_{V_{\theta}}(q) \right|^{1/2}. \end{align}
If $|a| \geq 4$ then $\left|a + 2p \wedge \pi_{V_{\theta}}(q) \right| \geq |a|/2$, and so by \eqref{case2},
\[ \eqref{lastexpr}  \gtrsim |z+\zeta|^{1/2} |z-\zeta|^{1/2}|a|^{1/2}   = |t-\tau-2z\wedge \zeta|^{1/2} \gtrsim  d_{\mathbb{H}}(v,w). \]
The argument is then similar to the case of \eqref{trivial}. Hence it may be assumed that $|a| < 4$. With this assumption, \eqref{case2} gives
\[ d_{\mathbb{H}}(v,w) \lesssim |t-\tau - 2z \wedge \zeta|^{1/2} \lesssim |z+\zeta|^{1/2}|z-\zeta|^{1/2}, \]
and putting this into \eqref{lastexpr} yields
\[ d_{\mathbb{H}}\left( P_{\mathbb{V}^{\perp}_{\theta}}(v), P_{\mathbb{V}^{\perp}_{\theta}}(w) \right) \gtrsim d_{\mathbb{H}}(v,w)\left|a + 2p \wedge \pi_{V_{\theta}} (q) \right|^{1/2}. \]
Therefore 
\begin{multline} \label{setinc} \left\{ \theta \in [0,\pi) : d_{\mathbb{H}}\left( P_{\mathbb{V}^{\perp}_{\theta}}(v), P_{\mathbb{V}^{\perp}_{\theta}}(w) \right) \leq \delta \right\} \\
\subseteq \left\{ \theta \in [0,\pi) : \left|a + 2p \wedge \pi_{V_{\theta}}(q)\right| \leq \left(\frac{K\delta}{d_{\mathbb{H}}(v,w)}\right)^2 \right\}, \end{multline}
for a sufficiently large constant $K$. Define $F=F_{p,q}$ by
\[ F(\theta) = a + 2p \wedge \pi_{V_{\theta}}(q), \]
so that 
\[ F'(\theta) = 2p \wedge \partial_{\theta}\pi_{V_{\theta}}(q), \quad  F''(\theta) = 2p \wedge \partial_{\theta}^2\pi_{V_{\theta}}(q). \]
Using $\pi_{V_{\theta}}(q) = \langle q, e^{i \theta} \rangle e^{i \theta}$ gives 
\[ \partial_{\theta} \pi_{V_{\theta}}(q) = \langle q, ie^{i \theta} \rangle e^{i \theta} + \langle q, e^{i \theta} \rangle i e^{i \theta}, \quad \partial_{\theta}^2 \pi_{V_{\theta}}(q) = 2\left( \langle q, i e^{i \theta} \rangle i e^{i \theta} - \langle q, e^{i \theta} \rangle e^{i \theta} \right). \]
Therefore $\partial_{\theta} \pi_{V_{\theta}}(q)$ and $\frac{1}{2}\partial_{\theta}^2 \pi_{V_{\theta}}(q)$ are orthonormal vectors in $\mathbb{R}^2$, for each $\theta \in [0,\pi)$, and so
\begin{equation} \label{oscillation} 1 = |p|^2 =  \left|\frac{F'(\theta)}{2} \right|^2 + \left|\frac{ F''(\theta) }{4}\right|^2 \quad \text{for all } \theta \in [0,\pi). \end{equation}
It follows that for any $b \in \mathbb{R}$, the equation $F(\theta)=b$ has at most 2 solutions in any interval of length strictly less than $1/2$. To see this, let $I$ be an interval with $|I| < 1/2$ and assume for a contradiction that $F(\theta)=b$ has three distinct solutions in $I$. Then by Rolle's Theorem $F'$ has two distinct zeroes in $I$, and by Rolle's Theorem again $F''$ has a zero $\theta''$ in $I$. Let $\theta'$ be one of the zeroes of $F'$. Then by \eqref{oscillation},
\[ 2 = |F'(\theta'') | = \left|\int_{\theta'}^{\theta''} F''(\theta) \, d\theta \right| \leq 4|I| < 2, \]
which is a contradiction. 

By covering the interval $[0,\pi)$ with 7 intervals of length strictly less than 1/2, the equation $F(\theta)=b$ has at most 14 solutions in $[0,\pi)$, for any $b$, and therefore the second set in \eqref{setinc} is the disjoint union of at most 15 subintervals of $[0,\pi)$. Equation \eqref{oscillation} implies that $F'$ is $4$-Lipschitz, so by using $\left\lfloor 8\pi \right\rfloor =25$, these at most 15 intervals can be written as a union of at most $15+25 = 40$ disjoint intervals $I \subseteq [0,\pi)$, each of length at most 1/8, such that either $|F'(\theta)| >1/2$ for every $\theta \in I$, or $|F''(\theta)|>3$ for every $\theta \in I$. Lemma 3.3 from \cite{christ} asserts that each of these intervals has length $\lesssim \frac{\delta}{d_{\mathbb{H}}(v,w) }$, assuming $d_{\mathbb{H}}(v,w) \geq \delta$. If $d_{\mathbb{H}}(v,w) \leq \delta$ the lemma holds trivially provided $C>\pi$, so this finishes the proof. \end{proof}

\begin{lemma} \label{compromise} Fix $s \in \left( 2, 4\right]$, let $\nu$ be a nonzero finite compactly supported Borel measure on $\mathbb{H}$ with $\sup_{\substack{ v \in \mathbb{H} \\
r>0 }} \frac{ \nu(B_{\mathbb{H}}(v,r))}{r^s} < \infty$, and fix $\kappa >  \max\left\{\frac{s}{2}, \frac{3(s-1)}{4-s^{-1}} \right\}$ with $\kappa< s$. Then there exist $\delta_0, \eta>0 $ such that
\[ \nu\left\{ v \in \mathbb{H}: \mathcal{H}^1\left\{ \theta \in [0, \pi) : P_{\mathbb{V}^{\perp}_{\theta} \#}\nu\left(B_{\mathbb{H}}\left(P_{\mathbb{V}^{\perp}_{\theta}}(v), \delta\right) \right) \geq \delta^{s-\kappa}\right\} \geq \delta^{\eta} \right\} \leq \delta^{\eta}, \]
whenever $\delta \in (0, \delta_0)$.  
\end{lemma}
\begin{proof} Choose 
\begin{equation} \label{etachoice} \eta  = \frac{1}{10^4}
\min\left\{ \kappa-\frac{s}{2}, \kappa-\frac{3(s-1)}{4-s^{-1}}\right\}, \end{equation}
which is strictly positive by the assumption on $\kappa$. The choice of $\delta_0$ will be made implicitly to eliminate implicit constants and ensure that various trivial inequalities, such as $|\log \delta| \leq \delta^{-\eta}$, hold for $\delta< \delta_0$. Fix $\delta \in (0,\delta_0)$ and let 
\[ Z =Z_{\delta} =  \left\{ v \in \supp \nu: \mathcal{H}^1\left\{ \theta \in [0, \pi) : P_{\mathbb{V}^{\perp}_{\theta} \#}\nu\left(B_{\mathbb{H}}\left(P_{\mathbb{V}^{\perp}_{\theta}}(v), \delta\right) \right) \geq \delta^{s-\kappa}\right\} \geq \delta^{\eta} \right\}. \]
For each $v \in Z$, let 
\[ H(v) =  \left\{ \theta \in [0, \pi) : P_{\mathbb{V}^{\perp}_{\theta} \#}\nu\left(B_{\mathbb{H}}\left(P_{\mathbb{V}^{\perp}_{\theta}}(v), \delta\right) \right) \geq \delta^{s-\kappa}\right\}. \]
The number of integer powers of 2 between $\delta^{1-100\eta}$ and the Korányi diameter of $\supp \nu$ is $\lesssim \left|\log \delta\right|$. Hence by the Frostman condition on $\nu$, for each $\theta \in H(v)$ there exists a dyadic number $t=t(\theta,v)$ in this range such that 
\begin{equation} \label{tcond} \nu\left(P_{\mathbb{V}^{\perp}_{\theta}}^{-1}\left(B_{\mathbb{H}}\left(P_{\mathbb{V}^{\perp}_{\theta}}(v), \delta\right) \right) \cap A_{\mathbb{H}}(v,t,2t) \right) \gtrsim \delta^{s-\kappa} \left| \log \delta\right|^{-1}, \end{equation}
where $A_{\mathbb{H}}(v,r,R)$ denotes the Korányi annulus in $\mathbb{H}$ centred at $v$ with inner radius $r$ and outer radius $R$:
\[ A_{\mathbb{H}}(v,r,R) = \left\{ w \in \mathbb{H} : r \leq d_{\mathbb{H}}(v,w) \leq R \right\}. \]
The set $H(v)$ can be partitioned into sets $H(v,t)$ according to the dyadic values $t=t(\theta,v)$. Since again there are $\lesssim |\log \delta|$ values, at least one set $H(v,t(v))$ must have $\mathcal{H}^1$-measure $\gtrsim \delta^{\eta} |\log \delta|^{-1}$. The set $Z$ can be similarly partitioned according to the dyadic values $t(v)$, so that there is a subset $Z' \subseteq Z$ corresponding to a single dyadic value $t$, with $\nu(Z') \gtrsim \nu(Z) |\log \delta|^{-1}$. 

Since $\delta^{\eta}$ is much smaller than $|\log \delta|^{-1}$, this shows that there exists $t \geq \delta^{1-100\eta}$ with $t \lesssim 1$ and a subset $Z' \subseteq Z$ satisfying
\[ \nu(Z') \geq \delta^{\eta}\nu(Z), \]
and 
\begin{equation} \label{measureh} \mathcal{H}^1(H'(v)) \geq \delta^{2\eta} \quad \text{ for all $v \in Z'$,} \end{equation}
where $H'(v)$ is defined for any $v \in \mathbb{H}$ by
\begin{equation} \label{Hprimedefn} H'(v) = \left\{\theta \in [0,\pi)
: \nu\left(P_{\mathbb{V}^{\perp}_{\theta}}^{-1}\left(B_{\mathbb{H}}\left(P_{\mathbb{V}^{\perp}_{\theta}}(v), 2\delta\right) \right) \cap A_{\mathbb{H}}(v,t,2t) \right) \geq \delta^{s-\kappa+\eta} \right\}. \end{equation}
This construction can be done in a way that ensures that $Z'$ is $\nu$-measurable. By inner regularity of $\nu$, $Z'$ can also be taken to be compact. 

Fix $v \in Z'$. Using \eqref{measureh}, choose three subsets $H_j'(v) \subseteq H'(v)$ separated $\pmod \pi$ by a distance of at least $\delta^{4\eta}$ from each other, each contained in an interval of length $\delta^{4\eta}$, and each with 1-dimensional measure at least $\delta^{8\eta}$. This can be done by partitioning $[0,\pi)$ into $\lesssim \delta^{-4\eta}$ intervals of length $\delta^{4\eta}$, choosing the 6 intervals with the largest intersection with $H'(v)$ (in terms of $\mathcal{H}^1$-measure), and then choosing 3 with gaps between them $\pmod \pi$. By compactness of $Z'$, this construction can be modified to ensure that for each $j$, the sets $H_j'(v)$ are piecewise constant in $v$ over some disjoint Borel cover of $Z'$ (this last observation is only needed to avoid a measurability issue later, and is the reason for the factor of $2\delta$ in \eqref{Hprimedefn}). 

 For each $j=1,2,3$ and each $v \in Z'$ and $v_j \in \mathbb{H}$, define $v \sim_j v_j$ if 
\[ v_j \in P_{\mathbb{V}^{\perp}_{\theta}}^{-1}\left(B_{\mathbb{H}}\left(P_{\mathbb{V}^{\perp}_{\theta}}(v), 2\delta \right) \right) \cap A_{\mathbb{H}}(v,t,2t), \quad \text{for some } \theta \in H_j'(v). \] Set
\begin{equation} \label{alphadefn} \alpha =  \frac{s-\kappa + 1000\eta}{s} \in (0,1). \end{equation}
A major part of the rest of the proof will consist in verifying the following inequality:
\begin{multline} \label{claimed} \nu(Z)t^3 \delta^{1000 \eta + 3(s-\kappa-1)} \leq \nu^4\Big\{ (v,v_1,v_2,v_3) \in Z' \times \left(\mathbb{H}^1\right)^3 : v \sim_j v_j \text{ for all $j$,} \\
d_E\left(\zeta_2, \ell(\zeta_1,\zeta_3) \right) \geq \delta^{\alpha}  \text{ if } |\zeta-\zeta_1|, |\zeta-\zeta_3| \geq t/2 \Big\}  \\
\lesssim  \max\left\{ t^{2s} \delta^{s/2}, t^{1+s}  \delta^{(1-\alpha)(s-1)- 1000 \eta} \right\},  \end{multline}
where $d_E$ refers to the Euclidean distance, $\ell(a,b)$ means the infinite line through $a$ and $b$, $\nu^4 = \nu \times \nu \times \nu \times \nu$ and $v=(\zeta,\tau) \in \mathbb{C} \times \mathbb{R}$. The lemma will then be derived by comparing the two outer parts of \eqref{claimed}. The piecewise constant property of the sets $H_j'(v)$ ensures that the set in \eqref{claimed} is Borel measurable. 

\begin{sloppypar} To prove the lower bound of \eqref{claimed}, cover the interval $[0,\pi)$ with disjoint intervals of length $\delta/t$, and fix $v \in Z'$, $j \in \{1,2,3\}$. Since $\mathcal{H}^1(H_j'(v)) \geq \delta^{8\eta}$, there are at least $t\delta^{8\eta-1}$ intervals $I_k=I_{k,j}$ intersecting $H_j'(v)$, so pick some $\theta_k= \theta_{k,j}$ in each intersection. Then
\begin{equation} \label{prelemma} \nu\left(P_{\mathbb{V}^{\perp}_{\theta_k}}^{-1}\left(B_{\mathbb{H}}\left(P_{\mathbb{V}^{\perp}_{\theta_k}}(v), 2\delta\right) \right) \cap A_{\mathbb{H}}(v,t,2t)\right) \geq \delta^{s-\kappa+\eta}, \end{equation}
for each $k$, which follows from $H_j'(v) \subseteq H'(v)$ and the definition of $H'(v)$ in \eqref{Hprimedefn}. \end{sloppypar}

For fixed $v \in Z'$, $v_1,v_3 \in \mathbb{H}$ with $v \sim_1 v_1$, $v \sim_3 v_3$, and $\frac{t}{2} \leq |\zeta-\zeta_1|, |\zeta-\zeta_3| \leq 2t$, it will be shown that
\begin{multline} \label{prelemma2} \nu\bigg\{ v_2 \in P_{\mathbb{V}^{\perp}_{\theta_k}}^{-1}\left(B_{\mathbb{H}}\left(P_{\mathbb{V}^{\perp}_{\theta_k}}(v), 2\delta\right) \right) \cap A_{\mathbb{H}}(v,t,2t) : \\
 d_E\left(\zeta_2, \ell(\zeta_1,\zeta_3) \right) \geq \delta^{\alpha} \bigg\} \geq \delta^{s-\kappa+2\eta}. \end{multline}
This will follow from \eqref{prelemma} combined with the claim that for 
\[ v \in Z', \quad v \sim_1 v_1, \quad v \sim_3 v_3, \quad \frac{t}{2} \leq |\zeta-\zeta_1|, |\zeta-\zeta_3| \leq 2t, \]
 the set
\[ E:= \left\{ v_2 \in  P_{\mathbb{V}^{\perp}_{\theta_k}}^{-1}\left(B_{\mathbb{H}}\left(P_{\mathbb{V}^{\perp}_{\theta_k}}(v),2 \delta\right) \right) \cap A_{\mathbb{H}}(v,t,2t)  : d_E\left(\zeta_2, \ell(\zeta_1,\zeta_3) \right) < \delta^{\alpha} \right\}, \]
 is contained in a Korányi ball of radius $\delta^{\alpha-100\eta}$, thereby contributing $\lesssim \delta^{(\alpha-100\eta) s}$ to the measure in \eqref{prelemma}, which is much smaller than $\delta^{s-\kappa+2\eta}$ by the definition of $\alpha$ in \eqref{alphadefn}. To prove that $E$ is contained in the required Korányi ball, it will first be shown that the projected set
\begin{multline*}  F:= \Big\{ \zeta_2 \in \mathbb{R}^2: (\zeta_2,\tau_2) \in A_{\mathbb{H}}(v,t,2t) \text{ for some } \tau_2 \in \mathbb{R}: \\
\left|\pi_{V_{\theta_k}^{\perp}}(\zeta_2-\zeta)\right| < 2\delta, \quad d_E(\zeta_2, \ell(\zeta_1,\zeta_3) ) < \delta^{\alpha} \Big\}, \end{multline*}
is contained in a Euclidean disc in $\mathbb{R}^2$ of radius $\delta^{\alpha-50\eta}$. To see this, fix $\zeta_2 \in F$ with $v_2 =(\zeta_2,\tau_2)$. Define
\[ \ell(\theta_k) = \{ \zeta+\lambda e^{i \theta_k} : \lambda \in \mathbb{R} \}, \]
so that by $\alpha < 1$ and the definition of $F$, 
\begin{equation} \label{Fcontain} F \subseteq \mathcal{N}_{\delta^{\alpha}}(\ell(\zeta_1,\zeta_3))  \cap \mathcal{N}_{\delta^{\alpha}}(\ell(\theta_k)). \end{equation}
 It suffices to contain the right hand side of \eqref{Fcontain} in a disc of radius $\delta^{\alpha-50\eta}$. This will be done by establishing a lower bound on the angle $\theta$ between the two lines. Let $\phi_1 \in H'_1(v)$ and $\phi_3 \in H'_3(v)$ be such that $\left|\pi_{V_{\phi_j}^{\perp}}(\zeta-\zeta_j)\right| < 2 \delta$ with $j \in \{1,3\}$, so that  $\phi_1$ and $\phi_3$ are $\delta^{4\eta}$-separated $\pmod \pi$, by construction of the sets $H_j'(v)$. Write
\[ \zeta_2 = \lambda_1 \zeta_1 + (1-\lambda_1)\zeta_3 + \lambda_2 i (\zeta_1 - \zeta_3), \quad \lambda_1, \lambda_2 \in \mathbb{R} \quad \text{with} \quad |\lambda_2||\zeta_1-\zeta_3| < \delta^{\alpha}. \]
If $x$ and $y$ are unit vectors in $\mathbb{R}^2$, the angle $\theta$ between them satisfies $\left|\sin \theta\right| = |x \wedge y |$. Moreover, \eqref{tcond} implies that $t \geq \delta^{\alpha-20\eta}$ by \eqref{alphadefn} and the Frostman condition on $\nu$. Therefore
\begin{align} \notag \left| (\zeta_1 - \zeta_3) \wedge \pi_{V_{\theta_k}}(\zeta_2-\zeta) \right| &\geq \left| (\zeta_1 - \zeta_3) \wedge (\zeta_2-\zeta) \right| - 8t\delta  \\
\label{wedgebound1} &\geq \left| (\zeta_1-\zeta) \wedge (\zeta_3-\zeta) \right| - t(4\delta^{\alpha}+8\delta)  \\
\notag &\geq \left| \pi_{V_{\phi_1}}(\zeta_1-\zeta) \wedge \pi_{V_{\phi_3}}(\zeta_3-\zeta) \right| - t(4\delta^{\alpha}+16\delta)  \\
\notag &\geq \frac{t^2}{16} \left|\sin(\phi_3 - \phi_1)\right| - t(4\delta^{\alpha}+16\delta) \\
\label{wedgebound2} &\geq \frac{ t^2\delta^{4\eta}}{32} - t(4\delta^{\alpha} + 16 \delta) \\
\notag &\geq t^2 \delta^{5 \eta}. \end{align} 
Hence the angle $\theta$ between the two lines in \eqref{Fcontain} satisfies
\[ \left|\sin \theta\right| =  \left| \frac{(\zeta_1 - \zeta_3)}{|\zeta_1 - \zeta_3|}  \wedge \frac{ \pi_{V_{\theta_k}}(\zeta_2-\zeta)}{\big| \pi_{V_{\theta_k}}(\zeta_2-\zeta) \big|} \right| \geq \delta^{6\eta}.  \]
It follows that the intersection
\[ \mathcal{N}_{\delta^{\alpha}}(\ell(\zeta_1,\zeta_3))  \cap \mathcal{N}_{\delta^{\alpha}}(\ell(\theta_k)), \]
and therefore $F$, is contained in a Euclidean disc of radius $\delta^{\alpha-50\eta}$ inside $\mathbb{R}^2$. 

For the set $E$, let $v_2,v_2' \in E$. By the triangle inequality for the Korányi metric,
\[ d_{\mathbb{H}}\left(P_{\mathbb{V}^{\perp}_{\theta_k}} (v_2), P_{\mathbb{V}^{\perp}_{\theta_k}}(v_2') \right) \leq 4\delta. \]
Considering each component of the Korányi distance separately gives
\[ \left| \pi_{V_{\theta_k}^{\perp}}(\zeta_2-\zeta_2') \right| \leq 4\delta, \]
and
\[ \left| \tau_2 - \tau_2' - 2 \pi_{V_{\theta_k}}(\zeta_2) \wedge \pi_{V_{\theta_k}^{\perp}}(\zeta_2) +  2 \pi_{V_{\theta_k}}(\zeta_2') \wedge \pi_{V_{\theta_k}^{\perp}}(\zeta_2') \right| \leq 16\delta^2. \]
By the identity $\zeta_2 \wedge \zeta_2' = \pi_{V_{\theta_k}}(\zeta_2) \wedge \pi_{V_{\theta_k}^{\perp}}(\zeta_2')- \pi_{V_{\theta_k}}(\zeta_2') \wedge \pi_{V_{\theta_k}^{\perp}}(\zeta_2)$ and the two preceding inequalities,
\begin{equation} \label{wedge} |\tau_2 - \tau_2' - 2 \zeta_2 \wedge \zeta_2' | \lesssim \delta. \end{equation} 
The Euclidean projection of $E$ down to $\mathbb{R}^2$ is contained in $F$, and therefore in a ball of radius $\delta^{\alpha-50\eta}$. Combining this with \eqref{wedge} results in
\begin{align*}  d_{\mathbb{H}}(v_2, v_2') &= \left(  |\zeta_2-\zeta_2'|^4 + |\tau_2-\tau_2' - 2\zeta_2 \wedge \zeta_2'|^2 \right)^{1/4} \\
\notag &\lesssim  \delta^{\alpha-50\eta}  + \delta^{1/2} \\
\notag &\lesssim \delta^{ \alpha - 50 \eta}, \end{align*}
for any $v_2, v_2' \in E$, by the definition of $\alpha$ in \eqref{alphadefn} and the choice of $\eta$ in \eqref{etachoice}. This shows that the Korányi diameter of $E$ is $\lesssim \delta^{\alpha-50\eta}$, and thus $E$ is contained in a Korányi ball of radius $\delta^{\alpha-100\eta}$. This implies \eqref{prelemma2} by \eqref{prelemma}, the Frostman condition on $\nu$ and the definition of $\alpha$. 

For each $j$ and each $v \in Z'$, Lemma~\ref{fasslerlemma} ensures that each set in \eqref{prelemma} intersects $\lesssim 1$ of the others, and therefore summing \eqref{prelemma} over $k$ gives
\begin{equation} \label{anyj} \nu\left\{ v_j \in \mathbb{H}: v \sim_j v_j \right\} \gtrsim t \delta^{100 \eta + s-\kappa-1}. \end{equation}
Similarly, if $v \sim_1 v_1$ and $v \sim_3 v_3$, summing \eqref{prelemma2} over $k$ gives
\begin{multline} \label{j=2} \nu\left\{ v_2 \in \mathbb{H}: v \sim_2 v_2 :  d_E\left(\zeta_2, \ell(\zeta_1,\zeta_3) \right) \geq \delta^{\alpha} \text{ if } |\zeta-\zeta_1|, |\zeta-\zeta_3| \geq t/2 \right\} \\
\gtrsim t \delta^{100 \eta + s-\kappa-1}. \end{multline}

Using \eqref{anyj}, \eqref{j=2} and Fubini's Theorem yields
\begin{align*} &\nu^4\Big\{ (v,v_1,v_2,v_3) \in Z' \times \left(\mathbb{H}^1\right)^3 : v \sim_j v_j \text{ for all $j$,} \\
&\qquad  d_E\left(\zeta_2, \ell(\zeta_1,\zeta_3) \right) \geq \delta^{\alpha} \text{ if } |\zeta-\zeta_1|, |\zeta-\zeta_3| \geq t/2 \Big\}  \\
&= \int\limits_{Z'
} \int\limits_{\{v_1: v \sim_1 v_1\}} \int\limits_{\{v_3: v \sim_3 v_3\}} \int\limits_{\substack{ \{v_2: v \sim_2 v_2 \text{ and} \\
d_E(\zeta_2 ,\ell(\zeta_1,\zeta_3)) \geq \delta^{\alpha} \text{ if} \\ |\zeta-\zeta_1|, |\zeta-\zeta_3| \geq t/2\}}} \, d\nu^4(v_2,v_3,v_1,v) \\
&\gtrsim \nu(Z') \left(t \delta^{100 \eta + s-\kappa-1} \right)^3 \\
&\geq \nu(Z) t^3 \delta^{301 \eta + 3(s-\kappa-1)}, \end{align*}
which implies the lower bound of \eqref{claimed}. 

For $v_1, v_2, v_3 \in \mathbb{H}$, let 
\begin{multline*} A=A(v_1,v_2,v_3) \\
:= \{ v \in Z' : v \sim_j v_j \text{ for all $j$, } d_E(\zeta_2, \ell(\zeta_1,\zeta_3)) \geq \delta^{\alpha} \text{ if } |\zeta-\zeta_1|, |\zeta-\zeta_3| \geq t/2\}. \end{multline*}
The upper bound of \eqref{claimed} will be obtained by bounding $\nu(A)$ and then integrating over $v_1,v_2,v_3$. 

 Fix $v \in A$. For each $j \in \{1,2,3\}$, the inequality
\[ d_{\mathbb{H}}\left(P_{\mathbb{V}^{\perp}_{\theta}} (v), P_{\mathbb{V}^{\perp}_{\theta}}(v_j) \right) \leq 2\delta \quad \text{for some } \theta \in H_j'(v), \]
 implies
\begin{equation} \label{wedge2} |\tau - \tau_j - 2 \zeta \wedge \zeta_j | \lesssim \delta, \end{equation}
by a calculation similar to the derivation of \eqref{wedge}. Hence if $\left|\tau-\tau_j - 2\zeta \wedge \zeta_j \right|^{1/2} \geq |\zeta-\zeta_j|$ for some $j \in \{1,2,3\}$, then $d_{\mathbb{H}}(v,v_j) \lesssim \delta^{1/2}$ since \eqref{wedge2} corresponds to the second component of the Korányi distance; see \eqref{secondcomp}. Therefore
\begin{equation} \label{firstcase} \nu\{ v \in A:  \left|\tau-\tau_j - 2\zeta \wedge \zeta_j \right|^{1/2} \geq |\zeta-\zeta_j| \text{ for some } j \in \{1,2,3\} \} \lesssim \delta^{s/2}. \end{equation}

\begin{sloppypar} It remains to bound $\nu(A')$, where 
\[ A' = \{ v \in  A: \left|\tau-\tau_j - 2\zeta \wedge \zeta_j \right|^{1/2} < |\zeta-\zeta_j| \text{ for all } j \in \{1,2,3\} \}. \]
Define $G: \mathbb{H} \to \mathbb{R}^3$ by
\[ G(\zeta,\tau) = \begin{pmatrix} \tau-\tau_1 -2\zeta \wedge \zeta_1  \\
\tau-\tau_2 -2\zeta \wedge \zeta_2\\
\tau-\tau_3 -2\zeta \wedge \zeta_3\end{pmatrix}, \quad \text{so} \quad  DG(\zeta,\tau) = DG = \begin{pmatrix} -2y_1 & 2x_1 & 1 \\
-2y_2 & 2x_2 & 1 \\
-2y_3 & 2x_3 & 1 \end{pmatrix}, \]
where $\zeta_j = x_j + i y_j$. Then $A' \subseteq G^{-1}(B_E(0,C\delta))$ for some constant $C$, by \eqref{wedge2}. If $v \in A'$, then $\frac{t}{2} \leq |\zeta-\zeta_j| \leq 2t$ for each $j \in \{1,2,3\}$ by definition of the Korányi metric. Hence if $A'$ is nonempty and there exists $v_0 \in A'$, then by the condition $d_E(\zeta_2, \ell(\zeta_1,\zeta_3)) \geq \delta^{\alpha}$ in the definition of $A$,
\[\zeta_2 = \zeta_3 + \lambda_1(\zeta_1-\zeta_3) + \lambda_2 i(\zeta_1-\zeta_3), \quad \lambda_1, \lambda_2 \in \mathbb{R} \quad \text{with} \quad |\lambda_2| |\zeta_1-\zeta_3| \geq \delta^{\alpha}. \]
The inequality $\left| (\zeta_1-\zeta) \wedge (\zeta_3-\zeta) \right| \gtrsim t^2\delta^{4\eta}$ follows similarly to the working from \eqref{wedgebound1} to \eqref{wedgebound2}, and this gives
\[  \left| (\zeta_1-\zeta) \wedge (\zeta_3-\zeta) \right| \gtrsim \delta^{4\eta} |\zeta_1 -\zeta||\zeta_3-\zeta|. \]
Using the identity $|z|^2 |w|^2 = |\langle z,w \rangle |^2 + |z \wedge w|^2$ for $z,w \in \mathbb{R}^2$ and expanding out $|(\zeta_1-\zeta)-(\zeta_3-\zeta)|^2$ gives $|\zeta_1-\zeta_3|^2 \geq t^2\delta^{10\eta}$. Hence
\begin{align*} \left|\det DG\right| &= 4|\zeta_1 \wedge \zeta_2 + \zeta_2 \wedge \zeta_3 + \zeta_3 \wedge \zeta_1| \\
&= 4\left| (\zeta_1-\zeta_3) \wedge (\zeta_2 - \zeta_3) \right| \\
&= 4|\lambda_2||\zeta_1-\zeta_3|^2 \\
&\geq 4t\delta^{\alpha + 5 \eta}. \end{align*}
By combining this with the formula $(DG)^{-1} = \left(\det DG \right)^{-1} \adj DG$ for the inverse, where $\adj$ refers to the adjugate, the operator norm satisfies $\left\|(DG)^{-1}\right\| \lesssim t^{-1}\delta^{-\alpha-5\eta}$. Hence
\begin{equation} \label{euclidball} A' \subseteq G^{-1}(B_E(0,C\delta)) \subseteq B_E\left(v_0, C't^{-1} \delta^{1-\alpha-5\eta}\right). \end{equation}

The condition on $t$ in \eqref{tcond} implies $t \geq \delta^{1/2}$, and so the radius of the ball in \eqref{euclidball} is less than 1 by the definitions of $\eta$ and $\alpha$ in \eqref{etachoice} and \eqref{alphadefn}. Proposition \ref{covering} therefore implies that $A'$ can be covered by $\lesssim t\delta^{-(1-\alpha-5\eta)}$ Korányi balls of radius $t^{-1} \delta^{1-\alpha-5\eta}$. Hence by the Frostman condition on $\nu$,
\begin{equation} \label{nuprime} \nu(A') \lesssim t^{1-s}\delta^{(1-\alpha-5\eta)(s-1)}. \end{equation}
Combining \eqref{nuprime} with \eqref{firstcase} therefore yields
\[ \nu(A) \lesssim \max\left\{ \delta^{s/2},  t^{1-s}\delta^{(1-\alpha-5\eta)(s-1)} \right\}. \]
By the triangle inequality and Fubini,
\begin{align*} &\nu^4\Big\{ (v,v_1,v_2,v_3) \in Z' \times \left(\mathbb{H}^1\right)^3 : v \sim_j v_j \text{ for all $j$,} \\
&\qquad d_E\left(\zeta_2, \ell(\zeta_1,\zeta_3) \right) \geq \delta^{\alpha}  \text{ if } |\zeta-\zeta_1|, |\zeta-\zeta_3| \geq t/2\Big\}\\
 &= \int_{\mathbb{H}} \int_{B_{\mathbb{H}}(v_3,4t)} \int_{B_{\mathbb{H}}(v_3,4t)}  \nu(A(v_1,v_2,v_3)) \, d\nu(v_1) \, d\nu(v_2) \, d\nu(v_3) \\
&\lesssim  \max\left\{ t^{2s} \delta^{s/2},  t^{1+s}  \delta^{(1-\alpha-5\eta)(s-1)}\right\}. \end{align*}
This implies the upper bound in \eqref{claimed}, which finishes the proof of \eqref{claimed}.  \end{sloppypar}

The inequality $\delta^{s/2} \leq \delta^{(1-\alpha-5\eta)(s-1)}$ follows from the definition of $\alpha$ in \eqref{alphadefn} and the definition of $\kappa$; if $\kappa = \frac{s}{2}+\epsilon$ this is trivial, and if $\kappa = \frac{3(s-1)}{4- s^{-1}}+\epsilon$ it follows from the inequality 
\[ 2s^2 -11s + 6 <0 \quad \text{for $s \in [1,4]$, } \]
where $\epsilon$ is sufficiently small, in either case. Hence the second term in the upper bound of \eqref{claimed} is actually always the maximum. Therefore, comparing the lower and upper bounds from \eqref{claimed} gives
\[ \nu(Z)t^3 \delta^{1000 \eta + 3(s-\kappa-1)} \lesssim   t^{1+s}  \delta^{(1-\alpha)(s-1)-1000\eta}. \] 
Using $t \lesssim 1$, this simplifies to
\begin{align*} \nu(Z) &\leq \delta^{ (1-\alpha)(s-1) -3(s-\kappa-1) - 2001 \eta } \\
&\leq \delta^{\eta}\delta^{ \left( 4-s^{-1} \right)\left( \kappa - \frac{3(s-1)}{4-s^{-1}}\right) - 4000 \eta} &&\text{using $\alpha = \frac{s-\kappa+1000\eta}{s}$,} \\
&\leq \delta^{\eta}, \end{align*}
by the definition of $\eta$ in \eqref{etachoice}. This finishes the proof of the lemma. \end{proof}

Finally, the proof of the main theorem can be given by combining Lemma~\ref{abstract} and Lemma~\ref{compromise} with Lemma~\ref{frostman} (Frostman).
\begin{proof}[Proof of Theorem~\ref{maintheorem}] Let $A \subseteq \mathbb{H}$ be an analytic set with $s :=\dim A >2$ and let $\epsilon \in (0,s/2)$. By Lemma~\ref{frostman} (Frostman), there is a nonzero, finite, compactly supported Borel measure $\nu$ on $A$ with $\nu(B_{\mathbb{H}}(v,r)) \leq r^{s-\epsilon}$ for every $v \in \mathbb{H}$ and $r>0$. By Lemma~\ref{compromise} with 
\[ \kappa = \max\left\{\frac{s}{2}, \frac{3(s-1)}{4-s^{-1}} \right\}> \max\left\{\frac{s-\epsilon}{2}, \frac{3(s-\epsilon-1)}{4-(s-\epsilon)^{-1}}\right\}, \]
there exist $\delta_0, \eta>0 $ such that
\[ \nu\left\{ v \in \mathbb{H}: \mathcal{H}^1\left\{ \theta \in [0, \pi) : P_{\mathbb{V}^{\perp}_{\theta} \#}\nu\left(B_{\mathbb{H}}\left(P_{\mathbb{V}^{\perp}_{\theta}}(v), \delta\right) \right) \geq \delta^{s-\epsilon-\kappa}\right\} \geq \delta^{\eta} \right\} \leq \delta^{\eta}, \]
whenever $\delta \in (0, \delta_0)$. The set $[0,\pi)$ is a compact metric space when naturally identified with a circle, and the vertical projections are continuous with respect to this metric. Therefore Lemma~\ref{abstract} gives
\[ \dim P_{\mathbb{V}^{\perp}_{\theta}} A \geq \dim P_{\mathbb{V}^{\perp}_{\theta}}( \supp \nu) \geq s-\epsilon-\kappa \quad \text{ for a.e.~$\theta \in [0,\pi)$.} \] 
Letting $\epsilon \to 0$ results in
\[\dim P_{\mathbb{V}^{\perp}_{\theta}} A \geq s-\kappa =  \begin{cases} \frac{s}{2} &\text{ if } s \in \left(2, \frac{5}{2} \right] \\
\frac{ s(s + 2)}{4s -1} &\text{ if } s \in \left( \frac{5}{2} , 4 \right], \end{cases} \]
for a.e.~$\theta \in [0,\pi)$.  \end{proof}

\end{document}